\documentclass[12pt,letterpaper]{amsart}
\usepackage{amsmath}
\usepackage{amsfonts}
\usepackage{float}

\usepackage[pdftex]{graphicx}
\usepackage{tikz}
\tikzset{partition/.style={fill,circle,inner sep=1pt}}

\usetikzlibrary{positioning}
\usetikzlibrary{decorations.pathreplacing}

\usetikzlibrary{matrix,arrows}
\usetikzlibrary{calc}
\tikzset{partition/.style={fill,circle,inner sep=1pt},
         part/.style={baseline=0,scale=0.5,bend left=45},
         partlabel/.style={below}}
\usetikzlibrary{shapes,arrows}
\tikzstyle{pnt}=[draw,ellipse,fill,inner sep=1pt]

\tikzstyle{opnt}=[ ]
\tikzstyle{pntt}=[draw,ellipse,fill,inner sep=0.5pt]
\tikzstyle{point}=[draw,ellipse,fill,inner sep=2pt]
\usepackage{color}

\usepackage{amssymb}
  \newtheorem{theorem}{Theorem}[section]
    \newtheorem{lemma}[theorem]{Lemma}

    \newtheorem{definition}[theorem]{Definition} 
    
    \newtheorem{example}[theorem]{Example} 
    \newtheorem{remark}[theorem]{Remark}

\newtheorem*{theorem*}{Theorem 4.1}

\newcommand{\Ass}{\operatorname{Ass}}

\begin{document}
\title[Associated Primes of Certain Graph Ideals]{Asymptotic Growth of Associated Primes of Certain Graph Ideals}
\author{Sarah Wolff}
\date{\today}
\begin{abstract} We specify a class of graphs, $H_t$, and characterize the irreducible decompositions of all powers of the cover ideals. This gives insight into the structure and stabilization of the corresponding associated primes; specifically, providing an answer to the question ``For each integer $t\geq 0$, does there exist a (hyper) graph $H_t$ such that stabilization of associated primes occurs at $n\geq (\chi(H_t)-1)+t$?" \cite{colorings}. For each $t$, $H_t$ has chromatic number $3$ and associated primes that stabilize at $n=2+t$. 
\end{abstract}
\maketitle

\section{Introduction}\label{intro}
We work at the intersection of graph theory and commutative algebra by relating powers of square-free monomial ideals in polynomial rings to vertex covers of finite simple graphs. Ideal constructions give two correspondences, both of which associate to each such graph a monomial ideal. These correspondences have been established and studied in, for example, \cite{villarreal}, \cite{simis}, \cite{oddholes}, and \cite{colorings}.

Motivated by the Strong Perfect Graph Theorem, Francisco, H\`{a}, and Van Tuyl \cite{oddholes} expand the link between graph theory and algebra in order to provide an algorithm for determining if a graph is \textit{perfect}--- neither it nor its complementary graph has an induced odd cycle of length five or greater. Their algorithm results from relating odd cycles to the associated primes of powers of monomial ideals. Francisco,  H\`{a}, and Van Tuyl continue to investigate this connection in \cite{colorings}, establishing a lower bound on the stabilization of these associated primes that is derived from the chromatic number of the corresponding graphs. They find that stabilization does not occur when $n< \chi(G)-1$, for $\chi(G)$ the chromatic number of the graph and $n$ the power of the ideal \cite[Corollary 4.9]{colorings}. They further show that this bound is not optimal, leading to the question: ``For each integer $t\geq 0$, does there exist a (hyper)graph $H_t$ such that the stabilization of associated primes occurs at $ n\geq (\chi(H_t)-1)+t $?" \cite[Question 4.10]{colorings}. In this paper we answer this question in the affirmative. 

More precisely, let $G$ be a graph with vertex set $V_G=\{v_1,\dots,v_m\}$ and edge set $E_G$ consisting of unordered pairs of distinct vertices of $G$. Let $K$ be a field. We identify the vertices of $G$ with the variables in the polynomial ring $R_G=K[v_1,\dots,v_m]$. Further, we associate two square-free monomial ideals to $G$: the \textit{edge ideal}, $I_G$, generated by $\{v_iv_j\;\vert\;\{v_i,v_j\}\in E_G\}$, and the \textit{cover ideal}, $J_G$, generated by $\{v_{i_1}\cdots v_{i_j}\;\vert\; \{v_{i_1},\dots,v_{i_j}\}\;\text{a vertex cover of}\; G\}$. These ideal constructions were first introduced in \cite{villarreal} and give the correspondences mentioned above.

Francisco,  H\`{a}, and Van Tuyl study the associated primes of $R_G/(J_G)^2$, denoted $\Ass(R_G/(J_G)^2)$, and find that $P$ is in $\Ass(R_G/(J_G)^2)$ if and only if either $P=(v_i,v_j)$ where $\{v_i,v_j\}$ is an edge of $G$, or $P=(v_{i_1},\dots,v_{i_s})$ where the subgraph induced by $\{v_{i_1},\dots,v_{i_s}\}$ is an odd cycle of $G$ \cite[Corollary 3.4]{oddholes}. These results are derived by determining the irreducible decomposition of $(J_G)^2$, i.e., writing $(J_G)^2$ as an intersection of irreducible monomial ideals.

They then extend these ideas in \cite{colorings} by studying $\Ass(R_G/(J_G)^n)$ for $n\geq 2$. Brodmann \cite{Brodmann} in 1979 proved that the set of associated primes of powers of ideals \textit{stabilizes}, i.e., for a ring $R$ and an ideal $I$ of $R$, there exists some positive integer $s$ such that $\Ass(R/I^s)=\Ass(R/I^n)$ for all $n\geq s$. Francisco, H\`{a}, and Van Tuyl \cite{colorings} find a lower bound on the stabilization of $\Ass(R_G/(J_G)^n)$, proving for a graph $G$ that stabilization does not occur when $n<\chi (G)-1$ \cite[Corollary 4.9]{colorings}. They then pose the question mentioned above \cite[Question 4.10]{colorings}, which we answer in the affirmative by providing a family of graphs, $H_t$, each with chromatic number three such that for each $t$, $\Ass(R_{H_t}/(J_{H_t})^n)$ stabilizes at $n=(\chi(H_t)-1)+t=2+t$.

The main result of this paper is:
\begin{theorem*}
For $t\in \mathbb{Z}_+$, let $R_{t}:=R_{H_t}$ be the polynomial ring associated to $H_t$ and let $J_t:=J_{H_t}$ denote the cover ideal of $H_t$. Then the associated primes of $R_t/(J_t)^n$ stabilize at $n=(\chi(H_t)-1)+t=2+t$. 
\end{theorem*}

As in \cite{oddholes}, we characterize the associated primes of $R_t/(J_{t})^n$ by studying the monomial ideal powers $(J_{t})^n$. Just as the cover ideal is generated by vertex covers of the graph, we find that the generators of the ideals in the irreducible decompositions of $(J_{t})^n$ are closely related to vertex covers. In Section \ref{background} we develop the necessary tools from algebra and graph theory for our results. In Section \ref{familyofgraphs} we introduce the family of graphs and explore the structure of these graphs with several key lemmas. In Section \ref{mainsection} we use these lemmas to characterize the irreducible decompositions of $(J_{t})^n$ (Theorem \ref{decomposition}). Theorem \ref{maintheorem} is then a direct consequence of Theorem \ref{decomposition}. 

\section{Graph Theory and Associated Primes}\label{background}
In this section we briefly introduce the important definitions from graph theory and algebra needed for our results. We use the basic conventions, notation, and definitions from  \cite{bollobas}, \cite{oddholes}, and \cite{miller}.

For a graph $G$ we enumerate its vertex set $V_G=\{v_1,\dots,v_m\}$. We assume that $G$ is a simple graph so that its edge set, $E_G$, consists of unordered pairs
$\{v_i,v_j\}$ of distinct vertices. Two vertices are \textit{adjacent} if they are joined by an edge.

A \textit{coloring} of $G$ is an assignment of colors to the vertices of $G$ such that no two adjacent vertices are assigned the same color. The \textit{chromatic number} of $G$ is the minimal number of colors in a coloring of $G$.

A subset $A$ of $V_G$ is a \textit{vertex cover} if every edge of $G$ is incident to at least one vertex of $A$. The subset is a \textit{minimum vertex cover} if it is a vertex cover of smallest cardinality. In this paper we consider minimum vertex covers rather than minimal vertex covers as this distinction is necessary for the proof of Theorem \ref{decomposition}.

For a subset $A$ of $V_G$, the \textit{neighbors} of $A$, denoted $N(A)$, are the vertices of $G$ that are adjacent to vertices of $A$, but do not lie in $A$. The subgraph of $G$ \textit{induced} by $A$ is the subgraph with vertex set $A$ and edge set consisting of those edges of $G$ with both vertices in $A$. 

A graph with vertex set $\{w_1,\dots,w_k\}$ is a \textit{$k$-cycle} if its edge set consists precisely of edges connecting $w_k$ to $w_1$ and $w_i$ to $w_{i+1}$, $1\leq i< k $. We call the cycle \textit{odd} (respectively, \textit{even}), if $k$ is odd (respectively, even).

Let $K$ be a field. We identify the vertices of $G$ with variables in the polynomial ring $R_G=K[v_1,\dots,v_m]$. Let $I_G$ denote the edge ideal of $G$, and let $J_G$ denote the cover ideal of $G$, generated by $\{v_{i_1}\cdots v_{i_j}\;\vert\; \{v_{i_1},\dots,v_{i_j}\}\;\text{a vertex cover of}\; G\}$. We note that $J_G$ is often equivalently defined as the Alexander dual of $I_G$ (see, e.g., \cite{oddholes}, \cite{colorings}). For an explanation of Alexander duality, see \cite{miller}. 

A typical monomial element of the cover ideal $J_G$ is of the form $v_1^{a_1}v_2^{a_2}\cdots v_m^{a_m}$ for nonnegative integers $a_i$. Just as the variables $v_1,\dots,v_m$ correspond to vertices of the graph $G$, the vector $(a_1,\dots,a_m)$ also has a graph theoretic interpretation. We refer to any such $\mathbf{a}=(a_1,\dots,a_m)$ (corresponding to the vertices $v_1,\dots, v_m$ of $G$) as a \textit{degree vector} and say a nonzero degree vector  $\mathbf{a}$ is a \textit{k-cover} if $a_i+a_j\geq k$ whenever $\{v_i,v_j\}\in E_G$. 

The concept of $k$-cover was first introduced in \cite{hibi} as a generalization of a vertex cover.  Note that if a subset $A$ of $V_G$ is a minimum vertex cover then a degree vector $\mathbf{a}$ with $a_i=1$ if $v_i\in A$ and $a_i= 0$ otherwise is a one-cover of $G$. We call such a one-cover a \textit{minimum one-cover}.

 



 An \textit{irreducible} monomial ideal is an ideal of the form $V^\mathbf{a}=(v_1^{a_1},\dots,v_m^{a_m})$ for $\mathbf{a}\in \mathbb{N}^n$, where if $a_i=0$, we omit $v_i^{a_i}$ from the set of generators. Every monomial ideal $I$ can be written as the finite intersection of irreducible monomial ideals: $I=V^{a_1}\cap \cdots\cap V^{a_s}$, called an \textit{irreducible decomposition} of $I$ \cite{miller}. The decomposition is called \textit{irredundant} if no $V^{a_i}$ can be omitted. 
 
Let $R$ be a commutative ring and $M$ an $R$-module. The set of \textit{associated primes} of $M$, denoted $\Ass(M)$, consists of prime ideals $P$ such that $P$ annihilates some $m$ in $M$. 

For a graph $G$, the cover ideal $J_G$ has irreducible decomposition $J_G=\bigcap_{\{v_i,v_j\}\in E_G} (v_i,v_j)$ (see e.g., \cite{oddholes}); hence, $\Ass(R/J_G)=\{(v_i,v_j)\;\vert\; \{v_i,v_j\}\in E_G\}$.  Generalizing to  $(J_G)^n$, we compute the irreducible decompositions of $(J_G)^n$ for our family of graphs in order to deduce the set of associated primes of $R/(J_G)^n$. We refer to this set as the \textit{associated primes of $G$}.

\section{A Family of Graphs}\label{familyofgraphs}

We proceed to define a family of graphs, $H_t$, to give an affirmative answer to the question posed in \cite[Question 4.10]{colorings}: ``For each integer $t\geq 0$, does there exist a (hyper)graph $H_t$ such that the stabilization of associated primes occurs at $ n\geq (\chi(H_t)-1)+t $?" Note by \cite[Remark 3.3, Corollary 3.4]{oddholes} that the associated primes of an odd cycle, $G$, stabilize at $n=\chi(G)-1+0=2$. We thus answer the above question for $t\geq 1$.

\begin{definition}\label{graphs} The graph $H_1$ has vertex set $\{x_1,\dots,x_5,y_1\}$ and edge set such that the subgraph induced by $\{x_1,\dots,x_5\}$ is a $5$-cycle and the neighbors of $y_1$ are precisely $\{x_1,x_2,x_3\}$.

For $t>1$, $H_t$ has vertex set $\{x_1,\dots,x_{4t-1},y_1,\dots,y_t\}$ and edge set such that the subgraph induced by $\{x_1,\dots,x_{4t-1}\}$ is a $(4t-1)$-cycle, the neighbors of $y_1$ are precisely $\{x_1,x_2,x_3\}$, and for $1<i\leq t$ the neighbors of $y_i$ are precisely $\{x_{4i-4},x_{4i-3},x_{4i-2},x_{4i-1}\}$  (Figure \ref{fig:1}).
\end{definition}
We first note that the graphs $H_t$ have chromatic number three for all $t\geq 1$. To see this, assign color $1$ to the vertices $\{x_1,x_3\}\cup\{y_i\;\vert\; i>1\}$, color $2$ to the vertices $\{x_2\}\cup\{x_{2k+1}\;\vert\; k>1\}$, and color $3$ to the remaining vertices. Each $H_t$ has at least one odd cycle; thus, this $3$-coloring is a minimum coloring. 
\begin{figure}[H]

\begin{tikzpicture}[scale=0.6]
\foreach \i in {0,...,4}
	{
		\node[pnt] at (0,\i)(\i){};
	}
	\foreach \i / \j in {0/5,1/4,2/3, 3/2, 4/1}
	{
		\draw (\i) node[left] {\footnotesize $x_{\j}$};
	}
\node[pnt] at (1, 3)(y1){};
\draw (y1) node[right] {\footnotesize $y_1$};
\draw(0)  to [bend left=45] (4);
\draw(4)  to (3);
\draw(3)  to (2);
\draw(2)  to (1);
\draw(1)  to (0);
\draw(4)  to (y1);
\draw(3)  to (y1);
\draw(2)  to (y1);

\begin{scope}[shift={(4,-2)}]
\foreach \i in {0,...,6}
	{
		\node[pnt] at (0,\i)(\i){};
	}
	\foreach \i / \j in {0/7,1/6,2/5, 3/4, 4/3, 5/2, 6/1}
	{
		\draw (\i) node[left] {\footnotesize $x_{\j}$};
	}
\node[pnt] at (1, 1.5)(y2){};
\draw (y2) node[right] {\footnotesize $y_2$};

\node[pnt] at (1, 5)(y1){};
\draw (y1) node[right] {\footnotesize $y_1$};

\draw(0)  to [bend left=35] (6);
\draw(6)  to (5);
\draw(5)  to (4);
\draw(4)  to (3);
\draw(3)  to (2);
\draw(2)  to (1);
\draw(1)  to (0);
\draw(6)  to (y1);
\draw(5)  to (y1);
\draw(4)  to (y1);
\draw(3)  to (y2);
\draw(2)  to (y2);
\draw(1)  to (y2);
\draw(0)  to (y2);
\end{scope}

\begin{scope}[shift={(8,-6)}]
\foreach \i in {0,...,10}
	{
		\node[pnt] at (0,\i)(\i){};
	}
	\foreach \i / \j in {0/11,1/10,2/9, 3/8, 4/7, 5/6, 6/5, 7/4, 8/3, 9/2, 10/1}
	{
		\draw (\i) node[left] {\footnotesize $x_{\j}$};
	}
\node[pnt] at (1, 1.5)(y3){};
\draw (y3) node[right] {\footnotesize $y_3$};

\node[pnt] at (1, 5.5)(y2){};
\draw (y2) node[right] {\footnotesize $y_2$};

\node[pnt] at (1, 9)(y1){};
\draw (y1) node[right] {\footnotesize $y_1$};

\draw(0)  to [bend left=25] (10);
\draw(10)  to (9);
\draw(9)  to (8);
\draw(8)  to (7);
\draw(7)  to (6);
\draw(6)  to (5);
\draw(5)  to (4);
\draw(4)  to (3);
\draw(3)  to (2);
\draw(2)  to (1);
\draw(1)  to (0);
\draw(10)  to (y1);
\draw(9)  to (y1);
\draw(8)  to (y1);
\draw(7)  to (y2);
\draw(6)  to (y2);
\draw(5)  to (y2);
\draw(4)  to (y2);
\draw(3)  to (y3);
\draw(2)  to (y3);
\draw(1)  to (y3);
\draw(0)  to (y3);
\end{scope}

%
%
%
%

\draw (0,5) node {\large $H_1$};
\draw (4,5) node {\large $H_2$};
\draw (8,5) node {\large $H_3$};
\draw (12,0) node {...};
\end{tikzpicture}
\caption{$H_t$}
\label{fig:1}
\end{figure}

Let $J_t:=J_{H_t}$ be the cover ideal of $H_t$ for each $t\geq 1$ and let $R_t:=R_{H_t}$ be the polynomial ring associated to $H_t$. We show that $\Ass(R_t/(J_t)^n)$ stabilizes at $n=2+t$. This result will follow from a characterization of the irreducible decompositions of $(J_t)^n$ for all $t\geq 1$ and $n\geq 1$. For this characterization, we use the structure of the graphs together with the following technical lemmas, all of which follow easily from the definitions:

\begin{lemma}\label{MinA} Let $G$ be a graph with vertex set $\{v_1,\dots,v_m\}$. Let $R_G$ be the polynomial ring associated to $G$, and let $J_G$ be the cover ideal of $G$. If $M$ is a minimal generator of $(J_G)^n$, then 
$$M\in \bigcap_{\{v_i,v_j\}\in E_G} (v_i^n,v_j)\cap (v_i^{n-1},v_j^2)\cap\cdots\cap (v_i^2,v_j^{n-1})\cap (v_i,v_j^n).$$
\end{lemma}

\begin{lemma}\label{ncovers} 
Let $G$ be a graph with vertex set $\{v_1,\dots,v_m\}$ and cover ideal $J_G$. Then the mononomial $v_1^{a_1}\cdots v_m^{a_m}$ is an element of $(J_G)^n$ if and only if $\mathbf{a}=(a_1,\dots,a_m)$ is an $n$-cover of $G$ that can be written as the sum of $n$ one-covers, $\mathbf{a_i}$, of $G$: $\mathbf{a}=\mathbf{a_1}+\cdots+\mathbf{a_n}$. 
\end{lemma}

We note that a degree vector, $\mathbf{a}$, for a graph $G$ can also be interpreted as a degree vector for a subgraph of $G$ induced by $A\subseteq V_G$ by considering only the $a_i$ corresponding to the vertices of $A$. We denote this subgraph degree vector by $\mathbf{a}\vert_A$.

\begin{lemma}\label{inducedsubgraphs} For a graph $G$, let $\mathbf{a}$ be a degree vector for $G$. Let $G_1,\dots,G_l$ be induced subgraphs of $G$ such that $E_G\subseteq E_{G_1}\cup\cdots\cup E_{G_l}$. Then $\mathbf{a}$ is a one-cover of $G$ iff $\mathbf{a}\vert_{G_i}$ is a one-cover of $G_i$ for all $i$, $1\leq i \leq l$.
\end{lemma}

Our interest is in studying $(J_t)^n$. By Lemma \ref{ncovers}, elements of $(J_t)^n$ correspond to $n$-covers of $H_t$. These $n$-covers can be written as the sum of $n$ one-covers of $H_t$, which in turn arise from one-covers of induced subgraphs, by Lemma \ref{inducedsubgraphs}. Since minimum vertex covers give rise to one-covers, it will be informative to look at the structure of minimum vertex covers of induced subgraphs of our family $H_t$.

\begin{lemma}\label{coveringcycles} A minimum vertex cover of an odd $k$-cycle consists of $\frac{k+1}{2}$ vertices.
\end{lemma}

\begin{lemma}\label{coveringsubgraphs} For the graph $H_t$, a minimum vertex cover of the induced subgraph on $\{y_1\}\cup N(y_1)$ consists of the vertices $\{x_2,y_1\}$. For $i>1$, a minimum vertex cover of the induced subgraph on $\{y_i\}\cup N(y_i)$ consists of exactly $3$ vertices, one of which is $y_i$.
\end{lemma}

\section{Stability of Associated Primes}\label{mainsection}

We restate our main theorem. Recall that $J_t$ is the cover ideal of $H_t$ and $R_t$ the polynomial ring associated to $H_t$.

\begin{theorem}\label{maintheorem}
Let $t\in \mathbb{Z}_+$. The associated primes of $R_t/(J_t)^n$ stabilize at $n=(\chi(H_t)-1)+t=2+t$. 
\end{theorem}

This result will follow from the determination of the irreducible decompositions of $(J_t)^n$ in Theorem \ref{decomposition}. Each component of the irreducible decomposition of $(J_t)^n$ is an irreducible monomial ideal of the form $(v_{i_1}^{a_1},\dots,v_{i_k}^{a_k})$, for $v_{i_j}$ vertices of $H_t$ (either $x$ or $y$ vertices). The powers of the variables are in fact determined by the structure of the subgraph induced by $\{v_{i_1},\dots,v_{i_k}\} $. Thus, we abuse terminology in the following way: for an ideal $(v_{i_1}^{a_1},\dots,v_{i_k}^{a_k})\subseteq R_t$, we refer to each variable $v_{i_j}$ as a \textit{vertex}, each power $a_j$ as a \textit{degree}, and $\mathbf{a}=(a_1,\dots,a_k)$ as a \textit{degree vector}. 

For any simple graph $G$ with vertex set $\{v_1,\dots,v_m\}$, let $J_G$ be the cover ideal. Then by \cite[Theorem 3.2]{oddholes}, the irredundant irreducible decomposition of $(J_G)^2$ is:
$$(J_G)^2= \bigcap_{\{v_i,v_j\}\in E_G} [(v_i^2,v_j)\cap (v_i,v_j^2)]\bigcap\bigcap_{\substack{\{v_{i_1},\dots,v_{i_k}\}\\\text{an odd cycle}}} (v_{i_1}^2,\dots,v_{i_k}^2),$$ where the last intersection is over all subsets $\{v_{i_1},\dots,v_{i_k}\}$ such that the subgraph induced by $\{v_{i_1},\dots,v_{i_k}\}$ is an odd cycle.

For our family of graphs, we generalize the above irreducible decomposition of $(J_t)^2$ to one for $(J_t)^n$, $n>2$, by inductively adding a minimum one-cover's worth of degrees to each degree vector appearing in the irreducible decomposition of $(J_t)^{n-1}$. In other words, if $V^\mathbf{c}=(v_{i_1}^{c_1},\dots,v_{i_k}^{c_k})$ is an ideal in the irreducible decomposition of $(J_t)^{n-1}$, and $\mathbf{b}$ is a minimum one-cover of the subgraph induced by $\{v_{i_1},\dots,v_{i_k}\}$, we show that $V^{\mathbf{c}+\mathbf{b}}$ is an ideal in the irreducible decomposition of $(J_t)^n$.

To be more explicit, if the induced subgraph on $\{v_{i_1},v_{i_2},v_{i_3}\}\subseteq V_{H_t}$ is a $3$-cycle, then by \cite[Theorem 3.2]{oddholes}, $(v_{i_1}^2,v_{i_2}^2,v_{i_3}^2)$ is a component of the irreducible decomposition of $(J_t)^2$. We lift this to $(J_t)^n$ by showing that $(v_{i_1}^{a_1},v_{i_2}^{a_2},v_{i_3}^{a_3})$ is a component of the irreducible decomposition of $(J_t)^n$ whenever $\mathbf{a}=(a_1,a_2,a_3)$ is a degree vector for the induced subgraph on $\{v_{i_1},v_{i_2},v_{i_3}\}$ obtained by adding $(n-2)$ minimum one-covers of this subgraph to the degree vector $\mathbf{c}=(2,2,2)$. This results in many different possibilities for $\mathbf{a}$: by Lemma \ref{coveringcycles}, a minimum one-cover of the subgraph is of form $(0,1,1)$,$(1,0,1)$, or $(1,1,0)$. Thus, adding $(n-2)$ minimum one-covers to $\mathbf{c}$ leads to $n\choose 2$ possibilities for $\mathbf{a}$.

Rather than writing out all these possibilities, we introduce below the terminology \textit{n-admissible} and $\hat{n}$-\textit{admissible} to distinguish degree vectors of specific induced subgraphs of $H_t$. As the irreducible decompositions of $(J_t)^n$ are known for $n\leq 2$, we work with $n> 2$ in the following definitions.

Consider an arbitrary subset of vertices, $V=\{v_{i_1},\dots,v_{i_k}\}\subseteq V_{H_t}$. If the subgraph induced by $V$ is an odd cycle, we call $V$ an \textit{induced odd cycle}. Consider the corresponding degree vector $\mathbf{a}=(a_1,\dots,a_k)$. 
\begin{definition}\label{defnadmiss} A degree vector $\mathbf{a}$ corresponding to an induced odd cycle $V$ is n-admissible if $\mathbf{a}= (2,\dots,2)+\mathbf{b_1}+\cdots+\mathbf{b_{n-2}}$ where each $\mathbf{b_i}$ is a minimum one-cover of the cycle.
\end{definition}

\begin{remark}\label{nadmiss} Note by Lemma \ref{coveringcycles} that for $V$ an induced odd $k$-cycle, if the corresponding degree vector $\mathbf{a}$ is $n$-admissible, then $$\displaystyle \sum_{i=1}^k a_i = 2k+(n-2)\big(\frac{k+1}{2}\big).$$
\end{remark}

Now fix a positive integer $r\leq t$ and consider a subgraph induced by vertices $V\cup Y_r=\{v_{i_1},\dots,v_{i_k},y_{j_1},\dots,y_{j_r}\}$ 
where $V\subseteq V_{H_t}$  is an induced odd cycle, $Y_r\subseteq \{y_1,\dots,y_t\}$ (the specified $y$-vertices in Definition \ref{graphs}), and $N(Y_r)\subseteq V$. We call such a subgraph an \textit{r-cluster}, or an \textit{r-cluster induced by $V\cup Y_r$}, when we want to specify the vertices. Note that in the definition of $r$-cluster, $V$ can be an induced odd cycle of any length on any of the vertices of $H_t$, so long as $N(Y_r)\subseteq V$. On the other hand, $Y_r$ must contain exactly $r$ vertices, and they must be $y$-vertices. 

Recall that we assume $n>2$ for the following definition.
\begin{definition}\label{defnhat} A degree vector $\mathbf{c}$ corresponding to an $r$-cluster induced by $V\cup Y_r$, is $\hat{n}$-admissible if $\mathbf{c}=\mathbf{d}+\mathbf{e}+\mathbf{f_1}+\cdots+\mathbf{f_{n-3}}$, where:
\begin{itemize}
\item[(1)]  $d_i= 
    \left\{
     \begin{array}{cc}
       2\;  & v_{j_i}\in V\setminus N(Y_r)\\
       3\;  & \text{otherwise}
     \end{array}
   \right.
 $, 
\item[(2)] $\mathbf{e}\vert_{Y_r\cup N(Y_r)}=(0,\dots,0)$,
\item[(3)] $\mathbf{e}\vert_{V\setminus N(Y_r)}$ is a minimum one-cover of the subgraph induced by $V\setminus N(Y_r)$,
\item[(4)] for all $1\leq i\leq n-3$, $\mathbf{f_i}\vert_{V\setminus N(Y_r)}$ is a minimum one-cover of the subgraph induced by $V\setminus N(Y_r)$,
\item[(5)] for all $1\leq i\leq n-3$ and $y_{j_l}\in Y_r$, $\mathbf{f_i}\vert_{\{y_{j_l}\}\cup N(y_{j_l})}$ is a minimum one-cover of the subgraph induced by $\{y_{j_l}\}\cup N(y_{j_l})$.
\end{itemize}
\end{definition}
\begin{example}\label{ex} Let $n=5$. Consider the following $2$-cluster of $H_4$: the subgraph induced by $$V\cup Y_2= \{x_1, x_2, x_3, x_4,y_2,  x_7, x_8, y_3, x_{11}, x_{12}, x_{13},x_{14},x_{15}, y_1, y_4\}.$$ Then $\mathbf{d}=(3,3,3,2,2,2,2,2,2,3,3,3,3,3,3)$, one of the two possibilities for $\mathbf{e}$ is $(0,0,0,0,1,0,1,0,1,0,0,0,0,0,0)$ and two of the six possibilities for each $\mathbf{f_i}$, $i\in\{1,2\}$, are  $(0,1,0,1,0,1,0,1,0,1,0,1,0,1,1)$ and $(0,1,0,0,1,0,1,0,1,0,1,1,0,1,1)$. Thus, the degree vector $$\mathbf{c}=(3,5,3,3,4,3,4,3,4,4,4,5,3,5,5)$$ is $\hat{5}$-admissible.
\end{example}

\begin{remark} Consider the subgraph induced by $\{x_1,x_2,x_3,y_1\}$. By Lemma \ref{coveringsubgraphs}, the minimum vertex cover of this subgraph is $\{x_2, y_1\}$. Thus, for an arbitrary $r$-cluster, induced by $V\cup Y_r$, if $y_1\in Y_r$, the corresponding degrees for $x_2$ and $y_1$ would be $1$ in each $\mathbf{f_j}$. By the same lemma, for $i>1$, a minimum vertex cover of the subgraph induced by $\{y_i\}\cup N(y_i)$ consists of $y_i$ and two vertices in $N(y_i)$; thus, if $y_i\in Y_r$, then in each $\mathbf{f_j}$ the corresponding degrees for $y_i$ and two vertices in $N(y_i)$ would be $1$. 
\end{remark} 

The following lemma is key to the proof of Theorem \ref{decomposition}.

\begin{lemma}\label{degreearg}
Let $r\geq 1$ and $n> r+1$. Consider an arbitrary $r$-cluster of $H_t$, induced by $V\cup Y_r$. Let $\mathbf{c}=(c_1,\dots,c_{k+r})$ be a degree vector for the $r$-cluster. If $\mathbf{c}$ is $\hat{n}$-admissible then $\sum_{i=1}^{k+r}c_i < r+k+n\big(\frac{k+1}{2}+r\big)$.
\end{lemma}
\begin{proof} Consider an arbitrary $r$-cluster of $H_t$, induced by $V\cup Y_r$, and suppose the corresponding degree vector, $\mathbf{c}$, is $\hat{n}$-admissible. We compute the degree sum imposed by the definition of $\hat{n}$-admissible. This degree sum depends on whether $y_1\in Y_r$, or not. 

First suppose $y_1\in Y_r$. If $\mathbf{c}$ is $\hat{n}$-admissible, then $\mathbf{c}=\mathbf{d}+\mathbf{e}+\mathbf{f_1}+\cdots+\mathbf{f_{n-3}}$, with each degree vector as specified in Definition \ref{defnhat}. We first note that $$\sum_{i=1}^{k+r}d_i=12+15(r-1)+2(k-4r+1).$$
Next, we claim $$\sum_{i=1}^{k+r} e_i \leq \frac{k-4r+1}{2}.$$ To see this, we find a bound on the number of vertices in a minimum vertex cover of the subgraph induced by $V\setminus N(Y_r)$. First note that $k-4r+1$ vertices comprise $V\setminus N(Y_r)$. Further, the subgraph of $H_t$ induced by $V\setminus N(Y_r)$ consists of disjoint subgraphs of $V$ induced by one or more subsets $V_1,\dots, V_l\subset V$. Let $v_\alpha$ be the number of vertices in $V_\alpha$, $1\leq \alpha\leq l$. If $v_\alpha$ is even, a minimum vertex cover of the subgaph induced by $V_\alpha$ consists of $\frac{v_\alpha}{2}$ vertices, while if $v_\alpha$ is odd, a minimum vertex cover consists of $\frac{v_\alpha-1}{2}$ vertices. Hence, $$\sum_{i=1}^{k+r} e_i =  \sum_{v_\alpha \;\text{odd}} \frac{v_\alpha-1}{2}+\sum_{v_\alpha \;\text{even}} \frac{v_\alpha}{2}\leq \frac{1}{2}\sum_{\alpha=1}^l v_\alpha= \frac{k-4r+1}{2}, $$ as claimed. 

Finally, by Lemma \ref{coveringsubgraphs} and the above claim we see that for each $\mathbf{f_j}$, $1\leq j\leq n-3$,  $$\sum_{i=1}^{k+r} f_{j_i}\leq 2+3(r-1)+\frac{k-4r+1}{2}.$$ 
Since the degree sum of $\mathbf{f_j}$, $\sum_{i=1}^{k+r}f_{j_i}$, is independent of $j$, we have: $$\sum_{i=1}^{k+r} c_i = \sum_{i=1}^{k+r} d_i + \sum_{i=1}^{k+r} e_i +(n-3) \sum_{i=1}^{k+r} f_{j_i},$$ for any $j\in \{1,\dots,n-3\}$. The inequality of the lemma then reduces to showing that $r+1< n$, which is true by assumption.

Next suppose $y_1\notin Y_r$. Here, $$\sum_{i=1}^{k+r}d_i=15r+2(k-4r),$$ and we claim
$$\sum_{i=1}^{k+r} e_i \leq \frac{k-4r-1}{2}.$$ To see this, note that in this case $k-4r$ vertices comprise $V\setminus N(Y_r)$. As above, the subgraph of $H_t$ induced by $V\setminus N(Y_r)$ consists of disjoint subgraphs of $V$ induced by subsets $V_1,\dots,V_l\subset V$. Further, because $V$ is an induced odd cycle, $k-4r$ is odd, so at least one of the subsets, without loss $V_1$, has an odd number of vertices, $v_1$. Hence, as above,  $$\sum_{i=1}^{k+r} e_i =  \sum_{v_\alpha \;\text{odd}} \frac{v_\alpha-1}{2}+\sum_{v_\alpha \;\text{even}} \frac{v_\alpha}{2}\leq \frac{1}{2}\big( v_1-1+\sum_{\alpha=2}^l v_\alpha\big)= \frac{k-4r-1}{2}, $$ as claimed.

Finally, by Lemma \ref{coveringsubgraphs} and the above claim we see  that for each $\mathbf{f_j}$, $1\leq j\leq n-3$,  $$\sum_{i=1}^{k+r} f_{j_i}\leq 3r+\frac{k-4r-1}{2}.$$ Again, the degree sum of $\mathbf{f_j}$, $\sum_{i=1}^{k+r} f_{j_i}$, is independent of $j$, and the inequality of the lemma reduces to showing $r+1< n$.
\end{proof}

We now define several ideals to simplify the statement of Theorem \ref{decomposition}. Fix a graph $H_t$ and a positive integer $n$. For an arbitrary set of vertices $V=\{v_{i_1},\dots,v_{i_k}\}$ of $H_t$  with corresponding degree vector $\mathbf{a}=(a_1,\dots,a_k)$, let $V^\mathbf{a}$ denote the ideal $(v_{i_1}^{a_1},\dots,v_{i_k}^{a_k})$. Again, we remind the reader of our convention to omit $v_{i_j}^{a_j}$ when $a_j=0$.
\begin{definition} $$ A_{t,n}=\bigcap_{\{v_i,v_j\}\in E_{H_t}} (v_i^n,v_j)\cap (v_i^{n-1},v_j^2)\cap\cdots\cap (v_i^2,v_j^{n-1})\cap (v_i,v_j^n).$$
\end{definition}
\begin{definition} $\displaystyle B_{t,n}=\bigcap_{\substack{\text{subgraph of}\; H_t\; \text{induced}\\ \text{by}\; V\; \text{an odd cycle},\\ \mathbf{a}\; n\text{-admissible}}}V^\mathbf{a}.$
\end{definition}
\begin{definition} 
$\displaystyle D_{t,n}^r= \bigcap_{\substack{\text{subgraph of}\; H_t\; \text{induced}\\\text{by}\; V\cup Y_r \text{an}\; r\text{-cluster},\\ \mathbf{c}\; \hat{n}\text{-admissible}}}({V\cup Y_r})^\mathbf{c}.$
\end{definition}

We now give the irreducible decompositions of $(J_t)^n$:

\begin{theorem}\label{decomposition}
Let $t\in \mathbb{Z}_+$ and let $J_t$ be the cover ideal of $H_t$. For $n>2$, $(J_t)^n$ has the following irredundant irreducible decomposition:
$$(J_t)^n=A_{t,n}\cap B_{t,n}\cap\bigcap_{r=1}^{n-2} D_{t,n}^r$$
\end{theorem}

\begin{proof} For $t\in\mathbb{N}$ and $n>2$, let $L$ be the ideal $A_{t,n}\cap B_{t,n}\cap\bigcap_{r=1}^{n-2} D_{t,n}^r$. Let $M$ be a minimal generator of $(J_t)^n$. We show $M$ is in $L$ by showing it is in each of the ideals $A_{t,n}$, $B_{t,n}$, and $\bigcap_{r=1}^{n-2} D_{t,n}^r$ separately.

First suppose $M\notin B_{t,n}$. Since $M$ is a minimal generator of the monomial ideal $(J_t)^n$, $M= v_{1}^{a_1}\cdots v_{m}^{a_m}$, for $\{v_1,\dots,v_m\}$ the vertices of $H_t$. By Lemma \ref{ncovers} the powers, $\mathbf{a}=(a_1,\dots,a_m)$, when viewed as a degree vector of $H_t$ form an $n$-cover that can be written as the sum of $n$ one-covers, $\mathbf{b_i}$, of $H_t$: $\mathbf{a}=\mathbf{b_1}+\cdots+\mathbf{b_n}$. Let $V=\{v_{j_1},\dots,v_{j_k}\}$ be an arbitrary induced odd cycle of $H_t$. Since $\mathbf{a}$ is an $n$-cover of $H_t$, $\mathbf{a}\vert_V$ is an $n$-cover of the odd cycle. Further, for $1\leq i \leq n$, since $\mathbf{b_i}$ is a one-cover of $H_t$, $\mathbf{b_i}\vert_V$ is a one-cover of the cycle. Then $\mathbf{a}\vert_V=\mathbf{b_1}\vert_V+\cdots+\mathbf{b_n}\vert_V$ with each $\mathbf{b_i}\vert_V$ a one-cover of the subgraph induced by $V$. Thus, $\mathbf{a}\vert_V=(a_{j_1},\dots,a_{j_k})$ is an $n$-cover of the cycle that can be written as the sum of $n$ one-covers. Then by Lemma \ref{coveringcycles}, $\sum_{i=1}^k a_{j_i} \geq n\big(\frac{k+1}{2}\big)$. However, $M\notin B_{t,n}$, so $M\notin V^\mathbf{u}$ for some $n$-admissible $\mathbf{u}$. Thus, in $M$ each variable $v_{i_j}\in V$ has degree less than $u_j$, so $$\sum_{i=1}^k a_{j_i}\leq \sum_{i=1}^k (u_i-1)= k+(n-2)\big(\frac{k+1}{2}\big),$$ by Remark \ref{nadmiss}. But $k+(n-2)\big(\frac{k+1}{2}\big)< n\big(\frac{k+1}{2}\big)$, so $\mathbf{a}\vert_V$ is not an $n$-cover of the cycle that can be written as the sum of $n$ one-covers, a contradiction.  Thus, $M\in B_{t,n}$.

Now suppose $M\notin \bigcap_{r=1}^{n-2} D_{t,n}^r$. Again, $M=v_1^{a_1}\cdots v_m^{a_m}$ is a minimal generator of $(J_t)^n$ so the powers, $\mathbf{a}=(a_1,\dots,a_m)$, when viewed as a degree vector of $H_t$ form an $n$-cover that can be written as the sum of $n$ one-covers, $\mathbf{b_i}$, of $H_t$:  $\mathbf{a}=\mathbf{b_1}+\cdots+\mathbf{b_n}$. For $r\leq n-2$, consider an arbitrary $r$-cluster of $H_t$, induced by $V\cup Y_r$. As above, $\mathbf{a}\vert_{V\cup Y_r}$ is an $n$-cover of the $r$-cluster and for $1\leq i\leq n$, $\mathbf{b_i}\vert_{V\cup Y_r}$ is a one-cover of the $r$-cluster, so $\mathbf{a}\vert_{V\cup Y_r}=(a_{l_1},\dots,a_{l_{k+r}})$ is an $n$-cover of the $r$-cluster that can be written as the sum of $n$ one-covers. Then by Lemma \ref{coveringsubgraphs}, $\sum_{i=1}^{k+r} a_{l_i}\geq n\big(\frac{k+1}{2}+r\big)$. However, $M\notin\bigcap_{r=1}^{n-2}D_{t,n}^r$, so $M\notin (V\cup Y_r)^\mathbf{c}$, for some $\hat{n}$-admissible $\mathbf{c}$. Thus, as in the previous case, $$\sum_{i=1}^{k+r} a_{l_i}\leq \sum_{i=1}^{k+r} (c_i-1).$$ But since $n\geq r+2$, Lemma \ref{degreearg} shows  $$\sum_{i=1}^{k+r} (c_i-1)< n\big(\frac{k+1}{2}+r\big),$$ so $\mathbf{a}\vert_{V\cup Y_r}$ is not an $n$-cover of the $r$-cluster that can be written as the sum of $n$ one-covers, a contradiction. Thus, $M\in \bigcap_{r=1}^{n-2} D_{t,n}^r$. 

We finally note by Lemma \ref{MinA} that $M\in A_{t,n}$. Hence, $(J_t)^n\subseteq L$.

Now let $P$ be a minimal generator of $L$. The ideal $L$ is an intersection of monomial ideals, so $P=v_{i_1}^{a_1}\cdots v_{i_m}^{a_m}$. By definition, $P\in A_{t,n}$; therefore, the powers $\mathbf{a}=(a_1,\dots,a_m)$, of the variables of $P$ when viewed as a degree vector of $H_t$ form an $n$-cover. By Lemma \ref{ncovers}, it suffices to show $\mathbf{a}$ can be written as the sum of $n$ one-covers of $H_t$. 

Consider an arbitrary one-cluster of $H_t$, induced by $V\cup Y_1$. Note that $H_t$ can be written as a union of one-clusters. Thus, by Lemma \ref{inducedsubgraphs}, it suffices to show that $\mathbf{a}\vert_{V\cup Y_1}$ can be written as a sum of $n$ one-covers of the one-cluster. Applying Lemma \ref{inducedsubgraphs} again to the subgraph induced by $V\cup  Y_1=\{v_{i_1},\dots,v_{i_k},y_{j_1}\}$, it suffices to show $\mathbf{a}\vert_V$ and $\mathbf{a}\vert_{\{y_{j_1}\}\cup N(y_{j_1})}$ can be written as a sum of $n$ one-covers of the subgraphs induced by $V$ and $\{y_{j_1}\}\cup N(y_{j_1})$, respectively. 

Let $V$ be an induced odd $k$-cycle of $H_t$. Since $P\in B_{t,n}$, $P\in V^\mathbf{u}$ for some $n$-admissible $\mathbf{u}=(u_1,\dots,u_k)$. Thus, we need to show that $\mathbf{u}$ can be written as the sum of $n$ one-covers of $V$. By Definition \ref{defnadmiss}, $\mathbf{u}= (2,\dots,2)+ \mathbf{b_1}+\dots+\mathbf{b_{n-2}}$ with each $\mathbf{b_i}$ a one-cover of $V$. Thus, we need only show that $(2,\dots,2)$ can be written as a sum of two one-covers of $V$, but this is clear because $(2,\dots,2)=(1,\dots,1)+(1,\dots,1)$, which are trivially one-covers of $V$.

Now consider the subgraph induced by $\{y_{j_1}\}\cup N(y_{j_1})$. Since $P\in \bigcap_{r=1}^{n-2} D_{t,n}^r$, $P\in (V\cup Y_1)^\mathbf{c}$, for some $\hat{n}$-admissible $\mathbf{c}$. Thus, we need to show $\mathbf{c}\vert_{\{y_{j_1}\}\cup N(y_{j_1})}$ can be written as a sum of $n$ one-covers of the subgraph induced by $\{y_{j_1}\}\cup N(y_{j_1})$. By Definition \ref{defnhat}, $\mathbf{c}= \mathbf{d}+\mathbf{e}+\mathbf{f_1}+\cdots+\mathbf{f_{n-3}}$ with each $\mathbf{f_i}$ a one-cover of the subgraph induced by $\{y_{j_1}\}\cup N(y_{j_1})$. Further, $d_i=3$ for each vertex in $\{y_{j_1}\}\cup N(y_{j_1})$, so we need only show that $(3,\dots,3)$ can be written as a sum of $3$ one-covers of the subgraph induced by $\{y_{j_1}\}\cup N(y_{j_1})$, but this is clear by Lemma \ref{coveringsubgraphs}. 

Thus, for an arbitrary one-cluster, induced by $V\cup Y_1$, $\mathbf{a}\vert_{V\cup Y_1}$ can be written as the sum of $n$ one-covers of the subgraph induced by $V\cup Y_1$. This proves that $\mathbf{a}$ can be written as the sum of $n$ one-covers of $H_t$.

Hence, $L\subseteq (J_t)^n$.\end{proof}

Our main theorem now follows as a corollary to Theorem 
\ref{decomposition}:

\begin{proof}[Proof of Theorem \ref{maintheorem}]
For $t\in\mathbb{Z}_+$ consider the graph $H_t$ and label the vertices $V_{H_t}=\{v_1,\dots,v_m\}=X\cup Y$, where $X$ (respectively, $Y$) represents the set of $x$-vertices (respectively, $y$-vertices) as in Definition \ref{graphs} . By Theorem \ref{decomposition}, for $n>2$ it follows that a prime $P$ is in $\Ass(R_t/(J_t)^n)$ if and only if:
\begin{itemize}
\item[(a)] $P=(v_i,v_j)$ where $\{v_i,v_j\}\in E_{H_t}$,
\item[(b)] $P=(v_{i_1},\dots,v_{i_k})$, where the subgraph induced by $\{v_{i_1},\dots,v_{i_k}\}$ is an odd cycle,

or,

\item[(c)] For $1\leq r\leq n-2$, $P=(v_{i_1},\dots,v_{i_k},y_{j_1},\dots, y_{j_r})$, where the subgraph induced by $\{v_{i_1},\dots,v_{i_k},y_{j_1},\dots, y_{j_r}\}$ is an $r$-cluster.
\end{itemize}
By definition, $H_t$ has exactly $t$ $y$-vertices, so $H_t$ has a $t$-cluster, namely the graph $H_t$ itself. Thus, by (c) above, the ideal $(v_1,\dots,v_m)$ is an element of $\Ass(R_t/(J_t)^n)$ for $n\geq t+2$, but not for $n<t+2$. Thus, the associated primes do not stabilize until at least $t+2$. Further, note that the ideals of (a) and (b) are elements of $\Ass(R_t/(J_t)^n)$ for all $n\geq 2$, so the only prime ideals not in $\Ass(R_t/(J_t)^{2+t})$ that could appear in $\Ass(R_t/(J_t)^n)$ for $n>2+t$ would be ideals that fall into category (c). However, for $1\leq r\leq t$, ideals whose generators correspond to the vertices of $r$-clusters are elements of $\Ass(R_t/(J_t)^{2+t})$, so the only new ideals in $\Ass(R_t/(J_t)^n)$ for $n>2+t$ would be those whose generators correspond to the vertices of  $r$-clusters for $r>t$. The largest $r$ for which $H_t$ has an $r$-cluster is $r=t$; thus, no such new ideals exist in $\Ass(R_t/(J_t)^n)$ for $n>2+t$. Thus, the associated primes of $H_t$ stabilize at $2+t$.
\end{proof}

\begin{remark} We note that our family of graphs can be generalized slightly. The proof of Theorem \ref{decomposition} used the definition $\hat{n}$-admissible in two ways--- to show that the degree vectors can be written as $n$ one-covers, and to use the degree argument of Lemma \ref{degreearg}. Both facts remain unchanged if the induced $(4t-1)$-cycle is instead an odd cycle of any length greater than $4t-1$ (or in the case of $H_1$, if the $5$-cycle is any odd cycle of length greater than $5$). Further, for $i>1$, $y_i$ could have more than $4$ neighbors, as long as these neighbors are distinct from $N(y_j), j\neq i$, and the degree argument of Lemma \ref{degreearg} is still satisfied.   
\end{remark}
\section{Acknowledgements}
Many of the results of this paper started as computer experiments using the program Macaulay 2 \cite{M2}. The author would like to thank Amelia Taylor for introducing her to the subject and for her invaluable support and guidance throughout the research and writing processes, Thomas Shemanske for his help in the editing process, and Adam Van Tuyl for a helpful discussion of Lemma \ref{ncovers}. We also thank an anonymous referee for their insights and suggestions for improvement. 
\bibliographystyle{amsplain}
\bibliography{projectbib}

\providecommand{\bysame}{\leavevmode\hbox to3em{\hrulefill}\thinspace}
\providecommand{\MR}{\relax\ifhmode\unskip\space\fi MR }
\providecommand{\MRhref}[2]{%
  \href{http://www.ams.org/mathscinet-getitem?mr=#1}{#2}
}
\providecommand{\href}[2]{#2}
\begin{thebibliography}{1}

\bibitem{bollobas}
B{\'e}la Bollob{\'a}s, \emph{Modern graph theory}, Graduate Texts in
  Mathematics, vol. 184, Springer-Verlag, New York, 1998. \MR{1633290
  (99h:05001)}

\bibitem{Brodmann}
M.~Brodmann, \emph{Asymptotic stability of {${\rm Ass}(M/I^{n}M)$}}, Proc.
  Amer. Math. Soc. \textbf{74} (1979), no.~1, 16--18. \MR{521865 (80c:13012)}

\bibitem{oddholes}
Christopher~A. Francisco, Huy~T{\`a}i H{\`a}, and Adam Van~Tuyl,
  \emph{Associated primes of monomial ideals and odd holes in graphs}, J.
  Algebraic Combin. \textbf{32} (2010), no.~2, 287--301. \MR{2661419
  (2011e:05133)}

\bibitem{colorings}
\bysame, \emph{Colorings of hypergraphs, perfect graphs, and associated primes
  of powers of monomial ideals}, J. Algebra \textbf{331} (2011), 224--242.
  \MR{2774655 (2012b:13054)}

\bibitem{M2}
Daniel~R. Grayson and Michael~E. Stillman, \emph{Macaulay2, a software system
  for research in algebraic geometry}, Available at
  \href{http://www.math.uiuc.edu/Macaulay2/}%
  {http://www.math.uiuc.edu/Macaulay2/}.

\bibitem{hibi}
J{\"u}rgen Herzog, Takayuki Hibi, and Ng{\^o}~Vi{\^e}t Trung, \emph{Symbolic
  powers of monomial ideals and vertex cover algebras}, Adv. Math. \textbf{210}
  (2007), no.~1, 304--322. \MR{2298826 (2007m:13005)}

\bibitem{miller}
Ezra Miller and Bernd Sturmfels, \emph{Combinatorial commutative algebra},
  Graduate Texts in Mathematics, vol. 227, Springer-Verlag, New York, 2005.
  \MR{2110098 (2006d:13001)}

\bibitem{simis}
Aron Simis, Wolmer~V. Vasconcelos, and Rafael~H. Villarreal, \emph{On the ideal
  theory of graphs}, J. Algebra \textbf{167} (1994), no.~2, 389--416.
  \MR{1283294 (95e:13002)}

\bibitem{villarreal}
Rafael~H. Villarreal, \emph{Cohen-{M}acaulay graphs}, Manuscripta Math.
  \textbf{66} (1990), no.~3, 277--293. \MR{1031197 (91b:13031)}

\end{thebibliography}
\end{document}